\documentclass[a4paper, 11pt, pdftex]{scrartcl}
\usepackage{aromas}
\usepackage{difftrees}
\usepackage{xcolor}
\usepackage{amsmath,amssymb,amsthm}

\newcommand{\tr}{\mathrm{tr}}
\newcommand{\sgn}{\mathrm{sgn}}
 \newcommand{\FF}{\mathcal{F}}
 \newcommand{\AAA}{\mathcal{A}}
 \newcommand{\TT}{\mathcal{T}}
\newcommand{\tensor}{\otimes}
\newcommand{\from}{\colon}
\newcommand{\one}{\boldsymbol{1}}
\newcommand{\RR}{\mathbb{R}}

\newcommand{\defeq}{\colon=}
\newcommand{\adj}{\mathrm{adj}}

\newcommand*\samethanks[1][\value{footnote}]{\footnotemark[#1]}
\renewcommand{\div}{\operatorname{div}}

\newcommand{\tonelabeled}{
\begin{tikzpicture}[dloop]
\loopy{1}
\node[below=0.1mm of A1]{$i$};
\end{tikzpicture}
}
\newcommand{\tfivelabeled}{
\begin{tikzpicture}[dloop]
\loopy{2}
 \node[dloop black node,right=of A1](B){};
 \draw (A1) edge (B);
 \node[below=0.1mm of A2]{$i$};
 \node[below=0.1mm of A1]{$j$};
 \node[below=0.1mm of B]{$k$};
\end{tikzpicture}
}

 \newcommand{\commentout}[1]{}
 
\theoremstyle{plain}
\newtheorem{theorem}{Theorem}

\newtheorem{lemma}{Lemma}
\newtheorem{remark}{Remark}

\newbox\mybox
\def\centerfigure#1{%
    \setbox\mybox\hbox{#1}%
    \raisebox{-0.4\dimexpr\ht\mybox+\dp\mybox}{\copy\mybox}%
}

\newbox\mybox
\def\cfig#1{%
\hspace*{-1mm}
    \setbox\mybox\hbox{#1}%
    \raisebox{-0.3\dimexpr\ht\mybox+\dp\mybox}{\copy\mybox}%
\hspace*{-1mm}
}

\theoremstyle{definition}
\newtheorem{example}{Example}
\newtheorem{conjecture}{Conjecture}

\title{Using aromas to search for preserved measures and  integrals in Kahan's method}
\author{Geir Bogfjellmo%
  \thanks{Department of Mathematics, Norwegian University of Life Sciences, N-1430 Ås, Norway}
   \and
  Elena Celledoni%
  \thanks{Department of Mathematical Sciences, Norwegian University of Science
    and Technology, N-7491  Trondheim, Norway}
 \and
 Robert I McLachlan%
 \thanks{Institute of Fundamental Sciences, Massey University, Palmerston North,  New Zealand}   \and
 Brynjulf Owren%
  \samethanks[2] \and
  G R W Quispel%
  \thanks{Department of Mathematics and Statistics, La Trobe University, Melbourne,  Victoria 3086, Australia}
 }
\begin{document}
\maketitle
\begin{abstract}
 The numerical method of Kahan applied to quadratic differential equations is known to often generate integrable maps in low dimensions and can in more general situations exhibit preserved measures and integrals. Computerized methods based on discrete Darboux polynomials have recently been used for finding these measures and integrals. However, if the differential system contains many parameters, this approach can lead to highly complex results that can be difficult to interpret and analyze. But this complexity can in some cases be substantially reduced by using aromatic series. These are a mathematical tool introduced independently by Chartier and Murua and by Iserles, Quispel and Tse. We develop an algorithm for this purpose and derive some necessary conditions for the Kahan map to have preserved measures and integrals expressible in terms of aromatic functions. An important reason for the success of this method lies in the equivariance of the map from vector fields to their aromatic functions. We demonstrate the algorithm on a number of examples showing a great reduction in complexity compared to what had been obtained by a fixed basis such as monomials.
\end{abstract}

\renewcommand{\div}{\operatorname{div}}

\newbox\mybox
\def\centerfigure#1{%
    \setbox\mybox\hbox{#1}%
    \raisebox{-0.4\dimexpr\ht\mybox+\dp\mybox}{\copy\mybox}%
}

\newbox\mybox
\def\cfig#1{%
\hspace*{-1mm}
    \setbox\mybox\hbox{#1}%
    \raisebox{-0.3\dimexpr\ht\mybox+\dp\mybox}{\copy\mybox}%
\hspace*{-1mm}
}

\newbox\mybox
\def\cfiglab#1{%
    \setbox\mybox\hbox{#1}%
    \raisebox{-0.6\dimexpr\ht\mybox+\dp\mybox}{\copy\mybox}%
}





\bibliographystyle{amsplain}





\section{Introduction}
In this paper we combine ideas from two apparently unrelated subfields of computational mathematics in order to obtain 
a new compact and equivariant way of characterizing certain preserved measures and integrals for the Kahan integration method. In the study of preserved measures and integrals of birational maps, the method of discrete Darboux polynomials has been employed recently in several works, see e.g. \cite{celledoni19udd,celledoni22dad}. On the other hand, there has been a recent interest in generalising the notion of B-series to include a larger set called aromatic series going back to \cite{chartier07pfi,iserles07bsm}.
Whereas the aromatic series have so far been used mostly for classification purposes and no-go theorems, we believe the present work is the first time these ideas have been used in a constructive manner to obtain specific objects such as preserved measures and integrals.

Kahan's method \cite{kahan93unm}, also known as the Hirota--Kimura discretization \cite{hirota00dot,kimura00dot}, is a numerical scheme 
which has received considerable attention in recent years for its remarkable numerical, geometric properties and for its ability to yield integrable discretisations when applied to a large class of integrable ordinary differential equations \cite{petrera11oio, petrera20hoc, celledoni14ipo}. 
For a quadratic system $\dot{x}=f(x)$ on $\mathbb{R}^n$, with $i$th component given as
$$
    \dot{x}_i  = f_i(x) = \sum_{j,k} a_{ijk}x_jx_k + \sum_j b_{ij} x_j + c_i,\ i=1,\ldots,n,
$$
letting $x_i\approx x_i(t_m)$ and $x'_i\approx x_i(t_{m+1})$, where $t_{m+1}-t_m=h$, the method takes the form
\begin{equation} \label{Kahan1}
     \frac{x_i'-x_i}{h} = \sum_{j,k} a_{ijk}\frac{x_jx_k'+x_j'x_k}{2}   + \sum_j b_{ij} \frac{x_j+x_j'}{2} + c_i,\ i=1,\ldots,n.
\end{equation}
Kahan's method is linearly implicit and it can be shown \cite{celledoni13gpo} that it has the form of a birational map, 
\begin{equation} \label{Kahan2}
   \frac{x'-x}{h} = (I-\frac{h}{2}f'(x))^{-1}f(x),
\end{equation}
thus it is well defined for sufficiently small values of the stepsize $h$. Petrera et al. \cite{petrera11oio}  applied the method of Kahan to a large number of integrable quadratic differential equations, showing that in several examples this yields a discrete integrable map.
In the sequel we shall refer to the map of $\mathbb{R}^n$, $\Phi_h: x\mapsto x'$ as the Kahan map.

Kahan's method has remarkable geometric properties when applied to quadratic vector fields, and its merits were discovered and analysed in several articles \cite{kahan97usf,sanzserna94aus,petrera11oio,celledoni10}. A fundamental property is that it is self adjoint, meaning that
 $\Phi_h^{-1}=\Phi_{-h}$. As we shall see later, this property also ensures that preserved measures are either even or odd in $h$.
In \cite{celledoni13gpo}, general expressions for a preserved integral and measure were found 
for the Kahan discretizations of systems with a cubic Hamiltonian and a constant Poisson bracket.
More precisely, for a system
\begin{equation} \label{Canonical}
        \dot{x} = J\nabla H(x) = f(x),
\end{equation}
with $H$ a cubic multivariate polynomial and $J$  a constant skew-symmetric matrix, the Kahan map has a preserved integral with the closed form expression
\begin{equation} \label{HtCanonical}
      \tilde{H}_h(x)  = H(x) + \frac{h}{3} \nabla H^T(x) (I-\frac{h}{2}f'(x))^{-1} f(x),
\end{equation}
and a preserved measure
\begin{equation} \label{pmCanonical}
    \frac{dx_1\wedge dx_2\wedge\cdots\wedge dx_n}{\det(I-\frac{h}{2}f'(x))}.
\end{equation}
In this example, the vector field is divergence-free and exact  flow preserves the standard volume form. In other examples, this need not be the case.
Several examples of integrable maps obtained with Kahan's method, including examples where $\div f\ne 0$ can be found in \cite{celledoni14ipo,petrera18nro}.

The Kahan map coincides with the Runge--Kutta method 
\begin{equation}
\label{RKformulation}
    \frac{x'-x}{h}=-\frac{1}{2}f(x)+2f\left(\frac{x+x'}{2}\right)-\frac{1}{2}f(x')
\end{equation}
restricted to quadratic vector fields in $\mathbb{R}^n$, \cite{celledoni13gpo}.
A powerful property of Runge--Kutta methods in general is that they are equivariant with respect to any affine transformation between two linear spaces, see
\cite{mclachlan16bsm} for a general discussion of affine equivariance in B-series methods.  The notion of aromas and aromatic series \cite{iserles07bsm,chartier07pfi}  is a generalisation of B-series, the aromatic series are indexed by graphs that may also include loops and each graph is called an aroma.  An aromatic function can be assigned to a vector field $f$ and an aroma, an example of such a function is the divergence, $\div(f)$.
The map from vector fields to aromatic functions is equivariant with respect to the affine group acting via pullbacks.
 In several examples, we have noticed that preserved measures and integrals of the Kahan method can be expressed in terms of aromatic functions. In this paper, we outline a method that can be used to search for preserved measures and integrals in terms of aromatic functions.  At present we can in principle determine all preserved measures whose density is of the form $1/P$ where $P$ can be expressed in terms of aromatic functions up to a prescribed order\footnote{All calculations are subject to complexity limitations of the computer algebra system used. Our present implementation uses aromas up to order 6, but this is not a limitation of the algorithm.}. 
 We provide several examples of its use where we show that such preserved measures can actually be determined.
 
 We also present a machinery for analysing this problem by adapting combinatorial methods treated in a general setting in \cite{iserles07bsm,bogfjellmo19aso,munthekaas16abs} to the case of Kahan's method. Finally, we present some necessary conditions on vector fields for Kahan's method to possess preserved measures and first integrals in the desired form of aromatic functions. 

To get an idea of this approach, consider first the preserved measure with (reciprocal) density function $\det (I-\frac{h}{2}f'(x))$ given by \eqref{pmCanonical}. By applying the Newton--Girard formula for symmetric polynomials, we find
$$
\det (I-\frac{h}{2}f'(x)) = P(r_1,\ldots,r_d)\quad \text{where}\ r_i = \text{Tr}(h^if'(x)^i),
$$
and where $P$ is some multivariate polynomial. Each term of this polynomial is an aromatic function. 

Writing density functions in terms of aromatic functions has some attractive properties. The density function of a preserved measure under the Kahan map is itself invariant under affine transformations of the underlying vector field. By using aromatic functions, the resulting expression is automatically affine equivariant. As a consequence, density functions of preserved measures take a much simpler form than would be the case if for instance a monomial basis had been used.

Furthermore, expressing the preserved measures and first integrals in terms of aromas highlights an interesting property: The preserved measures and first integrals thus obtained are related to functional dependencies between the derivatives of the vector field $f$. Specifically, if $\ker F_f=\ker F_{\tilde{f}}$ (see \eqref{eq: kerdef}) for two vector fields $f$ and $\tilde{f}$, then Kahan's method applied to the two vector fields has the same preserved measures and first integrals.

Although our findings show that not every preserved measure can be written in terms of aromatic functions, many important ones can. 
Naturally, the algorithm we propose can discover only measures that belong to the linear span of aromatic functions. 
But when successful this approach 
typically provides much simpler expressions for the measures with a transparent connection to the underlying ODE vector field, as opposed to what can be obtained using for instance a monomial basis.

Just to illustrate the power of the aromatic functions, we mention the inhomogenous Nambu system which we will discuss in detail in section~\ref{sec:examples}. By parametrizing this problem by means of a monomial basis our algorithm yields a density function with no fewer than 15806 terms. However, by using aromatic functions we can express this same function with only 7 terms.

\section{The method of Darboux polynomials} \label{sec:darboux}

Let $\Phi:\mathbb{R}^n\rightarrow\mathbb{R}^n$ be any map and suppose that there exist a function $C$ and a polynomial $P$ such that
\begin{equation} \label{darbouxdef}
    P\circ\Phi = C\cdot P,
\end{equation}
we then call $P$ a (discrete) Darboux polynomial, and $C$ is a corresponding cofactor. If there are two such Darboux polynomials, $P_1$ and $P_2$  with the same cofactor $C$, then clearly by \eqref{darbouxdef}
$$
\frac{P_1}{P_2}\circ\Phi = \frac{P_1\circ\Phi}{P_2\circ\Phi} = \frac{P_1}{P_2},
$$
so $\frac{P_1}{P_2}$ is a first integral of the map $\Phi$. Denote by $D\Phi$ the Jacobian of $\Phi$. If it so happens that $P$ is a Darboux polynomial with cofactor $C=\det D\Phi$, then 
$P$ is in fact the (reciprocal) density of a preserved measure, 
$$
  \frac{dx_1\wedge dx_2\wedge\cdots\wedge dx_n}{P}.
$$
Differentiating the Runge-Kutta representation of the Kahan map \eqref{RKformulation}, we easily obtain that $\det D\Phi_h$ is a rational function of the form 
\begin{equation} \label{detDphih}
\det D\Phi_h(x) = \frac{\det (I+\frac{h}{2}f'(\Phi_h(x)))}{\det(I-\frac{h}{2}f'(x))}.
\end{equation}
In \cite{celledoni19udd} a systematic approach for determining Darboux polynomials was proposed. The idea is to
first factor $\det D\Phi_h$ (over $\mathbb{Q}$) 
$$
    \det D\Phi_h(x) = \frac{\prod_i N_i(x)^{r_i}}{\prod_i D_i(x)^{s_i}},
$$
and then use these factors to form candidate cofactors
$$
    C(x) = \frac{\prod_i N_i(x)^{r_i'}}{\prod_i D_i(x)^{s_i'}},\quad r_i', s_i' \geq 0.
$$
For a fixed choice of cofactor $C(x)$, we let $P\in\mathbb{R}_p(x)$ be a multivariate polynomial of degree $p$ in the variables $x$ with real coefficients, which can be expressed in the basis $\{\mathbf{e}_k\}_{k=0}^N$,
\begin{equation} \label{dpansatz}
    P(x) = \sum_{k=0}^N P_k\mathbf{e}_k,
\end{equation}
where $N\leq\frac{(n+p)!}{n!p!}$. The basis elements $\mathbf{e}_k$ can for example be chosen to be monomials of the form $x_1^{i_1}\cdots x_n^{i_n}$ for non-negative integers $i_1,\ldots,i_n$. With this setup, equation \eqref{darbouxdef} is turned into a linear system of equations for the coefficients $P_k$ that can (in principle) always be solved by a finite algorithm.

In this paper, we shall always apply this method with $C =  \det D\Phi_h$, but rather than a monomial basis, we shall search for preserved measures in the linear span of certain functions called aromas introduced in the next section.

\section{Aromas and aromatic series}
Before describing aromas and aromatic series, we recall the related concept of
B-series \cite{butcher72, hairer06}.

Every Runge--Kutta method, and many other numerical methods for ODEs can be expanded in a series involving the vector field $f$ and its derivatives.

For example, the formulation of Kahan's method in Equation \eqref{Kahan2} can be expanded as a geometric series
\begin{equation} \label{Kahanseries}
 \Phi_h(x)=x+\sum_{k=0}^\infty \frac{h^k}{2^k} (f'(x))^k f(x).
\end{equation}

In general B-series are indexed by the set of rooted trees, and take the form
\[
\Phi_h(x)=x+\sum_{\tau \in\TT}\frac{h^{|\tau|}b(\tau)}{\sigma(\tau)}F(\tau)(x),
\]
where $\TT$ is the set of rooted trees, $|\tau|$ is the number of vertices in
$\tau$, $\sigma(\tau)$ is the symmetry coefficient of $\tau$, equal to the cardinality of the symmetry group
of $\tau$, and $F(\tau)$ is a vector field depending on $f$ and its derivatives. 
$b(\tau)$ are coefficients depending on the integrator.

Expanding $(I-hf'(x))^{-1}$ in a power series and substituting in \eqref{Kahan2}, it is easily seen that in the B-series \eqref{Kahanseries} of Kahan's method 
$b(\tau)$ is non-zero only on the tall trees:
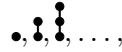
\[\ab, \aabb, \aaabbb, \dotsc,\]
In fact these represent the vector fields
\[
F(\ab)= f,\quad F(\aabb)=f'f,\quad F(\aaabbb)=f'f'f,\quad  \dotsc.
\]
For the tall trees, $\sigma(\tau)=1$ and we can see from \eqref{Kahanseries}
that Kahan's method has coefficients
\[
  b(\tau)= \begin{cases}
    2^{1-|\tau|} & \text{if $\tau$ is a tall tree,}\\
    0 & \text{otherwise.}
    \end{cases}
  \]

For the purpose of simplifying certain combinatorial formulas, it is useful to extend the definition of $b$ to multisets of rooted trees (called rooted forests) by multiplication, i.e.
\[
b(\tau_1\tau_2\dotsm\tau_m)\defeq b(\tau_1)b(\tau_2)\dotsm b(\tau_m).
\]

In B-series, vector fields depending on $f$ and its derivatives are represented by rooted trees. 
The measure we aim to preserve has an associated scalar valued density function which we assume depends on the vector field $f$ as well as its derivatives. We aim to describe such scalar functions by means of aromas (or loopy trees) and aromatic series.

Aromas and aromatic series were originally introduced by Iserles, Quispel and
Tse \cite{iserles07bsm} and by Chartier and Murua \cite{chartier07pfi}
and their structure was investigated by Munthe-Kaas and Verdier \cite{munthekaas16abs} by
Bogfjellmo \cite{bogfjellmo19aso}, and by Laurent et al. \cite{laurent2023aromatic}

An aroma is a connected directed graph where each vertex has exactly one
outgoing edge.
It can be shown that an aroma has to contain exactly one cycle.

The smallest aromas are
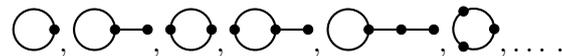
\[
\tone, \ttwo, \tthree, \tfive, \tsix, \tseven, \dotsc.
\]
To simplify graphics, directions of edges are not shown unless they are necessary to distinguish between aromas. The edges are oriented so that the ring is a cycle and other edges are oriented towards the ring.

We will refer to the set of aromas as $\AAA'$ and the set of multisets
(products) of aromas as $\AAA$.
The empty multiset will be denoted by $\one$.

Given a vector field $f$, an aroma $\lambda$ represents a scalar function
$F(\lambda)$ 
according to the following procedure:
\begin{enumerate}
\item Label each node $i,j,\dotsc$.
\item For each node with label $j$, form the factor $f^j_{i_1i_2\dotsm i_m}$, where
$i_1,i_2,\dotsc,i_m$ are the labels of the nodes pointing towards node $j$.
The upper index on $f$ corresponds to vector components, and the lower to
partial derivatives with respect to coordinate directions, i.e.
$f^j_{i_1i_2\dotsm i_m}=\partial^m f^j/\partial x_{i_1} \partial x_{i_2}\dotsm \partial x_{i_m}$.
\item Finally, take the product of the factors and sum all terms using Einstein's
summation convention.
\end{enumerate}

The definition of $F$ extends without modification to multisets of aromas,
viewed as disjoint unions of connected graphs.
\begin{example}
  Some examples of $F$.
    \begin{align*}
    F(\one )&=1 \\
    F\big(\cfiglab{\tonelabeled}\big)&= \sum_i f^i_i = \div(f) \\
    F\big(\cfiglab{\tfivelabeled}\big)& = \sum_{ijk} f^i_j f^j_{ik} f^k
\end{align*}
\end{example}

The simplest aromas are the \emph{cyclic aromas},
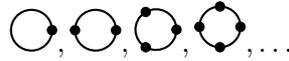
\[ \tone,\tthree, \tseven,\tsixteen, \dotsc\]
whose images under $F$ are traces of powers of $f'$.

An aromatic series is a series indexed by $\AAA$.
We will normalize these series as
\[B(\gamma)\defeq \sum_{\alpha\in \AAA}\frac{h^{|\alpha|}\gamma(\alpha)}{\sigma(\alpha)}F(\alpha),\]
where $|\alpha|$ is the number of vertices in $\alpha$ and $\sigma(\alpha)$ is the cardinality of the symmetry group of the graph
$\alpha$.

\begin{example}
  Some examples of $\sigma$ are
  \[\sigma(\one)=1,\]
  \[
    \sigma\big(\centerfigure{\tfive}\big) = 1,
  \]
  \[\sigma\big(\centerfigure{\tseven}\big)= 3,\]
  \[
    \sigma\big(\centerfigure{\ttwentyfour}\big) = 8.
  \]
\end{example}
In several cases, Darboux polynomials for the discretization of a vector field
with Kahan's method are expressible as finite aromatic series.

\begin{example}[Hamiltonian vector fields]
Assume that Kahan's method is used to discretize a Hamiltonian vector field with a
cubic Hamiltonian. It follows from the results of \cite{celledoni13gpo} that
$P(x)=\det\left(I-\frac{h}{2}f'(x)\right)$ is a Darboux polynomial with cofactor $\det
  D\Phi$.
  
 For any fixed dimension, $P$ can be expressed as a polynomial of traces. For example, when $A$ is a $2\times 2$-matrix, we have the identity
  \[\det(I-A) = 1 - \tr(A) + \frac{1}{2}\left(\tr(A)^2-\tr(A^2)\right).\]
  
  Thus, when $d=2$,
  \[
    \begin{aligned}
      P &=\det\left(I-\frac{h}{2}f'\right)\\
        &= 1-\frac{h}{2}\tr(f')+\frac{h^2}{8}\left(\tr(f')^2-\tr((f')^2)
      \right)\\
        &= F(\one) -\frac{h}{2} F\big(\centerfigure{\tone}\big) + \frac{h^2}{8} \left(F\big(\centerfigure{\tone \tone}\big)- F\big(\centerfigure{\tthree}\big)\right).
      \end{aligned}
   \]
 \end{example}
We refer to Section~\ref{sec:examples} for a detailed study of examples of the use of this approach.

\subsection{Equivariance}
One strength of the aromatic approach is that the resulting preserved measures are obtained through an affine equivariant map on the space of quadratic vector fields.

The evolution of a differential equation $\dot{x}=f(x)$ on $\mathbb{R}^n$ does not depend on the coordinates we use on $\mathbb{R}^n$. If we limit coordinate changes to affine mappings, neither does Kahan's method. Consequently, if Kahan's method applied to a differential equation $\dot{x}=f(x)$ has a preserved measure, then an affine change of coordinates $x\mapsto Ax+b$ will result in a new differential equation and, when Kahan's method is applied, a new discrete dynamical system that also has a preserved measure. In coordinates, this measure will be different from the original, but it's expression in terms of aromas stays the same.

We formulate this independence of coordinates as equivariance.
Equivariance can be defined\footnote{A slightly different and stronger definition of equivariance was used in \cite{mclachlan16bsm}} in the category of $G$-sets, the collection of sets on which a common Lie group $G$ is acting.
Let $X_1$ and $X_2$ be two $G$-sets. A map $\Phi:X_1\rightarrow X_2$, is said to be \emph{equivariant} if it is a $G$-set morphism, i.e. for any $g\in G$ and $x\in X_1$, $\Phi(g\cdot x)=g\cdot\Phi(x)$.

In our setting, the group is the affine group, consisting of pairs $g=(A_g,b_g)$ where $A_g\in GL(n,\mathbb{R})$ and $b_g\in \mathbb{R}^n$. This group acts on $\mathbb{R}^n$ through $g\cdot x=A_g x+b_g$. It also acts on $\mathcal{X}(\mathbb{R}^n)$ and $\mathcal{F}(\mathbb{R}^n)$, the vector fields and functions on $\mathbb{R}^n$ respectively. The (right) action is by pullback, or more precisely, we have for $f\in\mathcal{X}(\mathbb{R}^n)$, $\varphi\in\mathcal{F}(\mathbb{R}^n)$ respectively
$$
   g\cdot f(x) = A_g^{-1}f(A_g x+b_g),\quad g\cdot\varphi(x) = \varphi(A_gx+b).
$$
\begin{lemma}
Let $F:\mathcal{X}(\mathbb{R}^n)\times\mathcal{A}\rightarrow \mathcal{F}(\mathbb{R}^n)$ be the map that assigns to a vector field $f$ and an aroma $a$ the aromatic function $F(f,a)$. Then, for any $a\in\mathcal{A}$, the map $F(\cdot,a):\mathcal{X}(\mathbb{R}^n)\rightarrow\mathcal{F}(\mathbb{R}^n)$ is equivariant with respect to the affine group. 
\end{lemma}
\noindent The proof is omitted, and is based on similar results in  \cite{munthekaas16abs}.
\begin{lemma} \label{lemma36}
Let $P_f$ be the (reciprocal) density of a preserved measure for the Kahan map. Then $P_{g\cdot f}$ will be the (reciprocal) density of a preserved measure for Kahan's method applied to the vector field $g\cdot f$ for any affine transformation $g$.
\end{lemma}
\noindent{\em Proof:} Follows from \eqref{darbouxdef} and  \eqref{detDphih}.

\subsection{Algebraic treatment}
We now return to Darboux polynomials, given by the equation \eqref{darbouxdef}. We restrict our attention to $\Phi$ being the Kahan map \eqref{Kahan1}, $C=\det D\Phi$ and $P$ an aromatic series.

For the Kahan map, we have that
\[
\det( D\Phi_h(x)) =
\frac{\det(I+\frac{h}{2}f'(\Phi_h(x)))}{\det(I-\frac{h}{2}f'(x))}=\frac{N_{\frac{h}{2}}\circ
\Phi_h(x)}{N_{-\frac{h}{2}}(x)},
\]
with $N_{\frac{h}{2}}(x)=\det (I+\frac{h}{2}f'(x))$.
Consequently, the Darboux equation \eqref{darbouxdef} for a Darboux polynomial
$P$, can be written as
\begin{equation}
  N_{-\frac{h}{2}}(x)\cdot P\big(\Phi_h(x)\big)-P(x)\cdot N_{\frac{h}{2}}\big(\Phi_h(x)\big)=0.
  \label{darbouxkahan}
\end{equation}
Our goal is to express this equation in terms of aromatic series. 
We are therefore interested in: 
\begin{enumerate} \item expressing functions of the form
$g_0(x)g_1(x)$, where both factors are aromatic series (multiplication),
\item expressing functions of the form
$g(\Phi(x))$ where $g$ is an aromatic series and $\Phi(x)$ is a B-series update map,
specifically the Kahan map (composition), and
\item expressing $N_{uh}(x)=\det(I+uhf')$ as an aromatic series.
\end{enumerate}

\subsubsection{Multiplication and composition}
The necessary combinatorial formulas can be deduced from the results in
\cite{bogfjellmo19aso}.

For expressing the formulas, it is convenient to define the free vector spaces
generated by $\AAA$ and $\FF$ over the field of reals,
$\RR\langle\AAA \rangle$ and $\RR\langle \FF \rangle$, their duals
$\RR\langle\AAA \rangle^\ast$, $\RR\langle \FF \rangle^\ast$ and tensor products
thereof.

Furthermore, we introduce the notation $\langle\cdot,\cdot \rangle$ for pairing
elements in dual and primal spaces, especially in the case of tensor products:
Let $U$ and $V$ be vector spaces, $a\in U^\ast$, $b\in V^\ast$ and $w\in U\tensor V$ where $w$ has the decomposition
$w=\sum_{i=1}^nu_i\tensor v_i$, then
\[\langle a\tensor b, w \rangle= \sum_{i=1}^n a(u_i)b(v_i).\]

The product of two aromatic series is simply a product of formal power series in
the variables $\alpha \in \AAA'$ with a normalization factor\footnote{To be
  precise, exponential power series in $\{ \frac {\alpha}{\sigma(\alpha)} \colon
  \alpha\in \AAA'\}$.}.
\begin{lemma}[Multiplication]
  Let
  $g_i(x)=\sum_{\alpha\in\AAA}\frac{h^{|\alpha|}\gamma_i(\alpha)}{\sigma(\alpha)}F(\alpha)(x)$,
  for $i=0,1$.
  Then
  \[
    g_0(x)g_1(x)=\sum_{\alpha\in\AAA} \frac{h^{|\alpha|} \gamma_0\cdot \gamma_1(\alpha)}{\sigma(\alpha)}F(\alpha)(x)
  \]
  where $\gamma_0\cdot \gamma_1(\alpha)=\langle  \gamma_0\tensor
\gamma_1,\Delta_{\sqcup}(\alpha)\rangle$,
and $\Delta_\sqcup\from \RR\langle  \AAA \rangle \to \RR\langle  \AAA
\rangle\tensor  \RR\langle  \AAA \rangle$ is the binomial coproduct \cite[Section V.2]{joni79cab},
\[\Delta_\sqcup(\alpha)=\sum_{\beta\subseteq \alpha} \beta \tensor (\alpha \setminus \beta),\]
where the sum is over all submultisets of the multiset $\alpha$, counting multiplicities.
\label{lem: mult}
\end{lemma}

\begin{example}
  \begin{align*}\Delta_{\sqcup}\big(\centerfigure{ \tone \tone \ttwo }\big)=&
     \one  \tensor \centerfigure{ \tone \tone \ttwo } + 2 \centerfigure{\tone} \tensor \centerfigure{\tone \ttwo} \\
  &+ \centerfigure{ \ttwo } \tensor \centerfigure{ \tone \tone }+ \centerfigure{ \tone \tone} \tensor \centerfigure{ \ttwo} \\
    &+2\centerfigure{ \tone \ttwo } \tensor \centerfigure{ \tone} + \centerfigure{ \tone \tone \ttwo} \tensor  \one.
  \end{align*}
\end{example}

The composition is given by \cite[Theorem 5.1]{bogfjellmo19aso}, with
appropriate restrictions, as we are only interested in composing a B-series and
an aromatic series.
\begin{lemma}[Composition]
Let
$g(x)=\sum_{\alpha\in\AAA}\frac{h^{|\alpha|}\gamma(\alpha)}{\sigma(\alpha)}F(\alpha)(x)$
and
$\Phi(x)=x+\sum_{\tau \in\TT}\frac{h^{|\tau|}b(\tau)}{\sigma(\tau)}F(\tau)(x)$.

Then
\[g(\Phi(x)) =\sum_{\alpha\in\AAA}\frac{h^{|\alpha|}(b \odot \gamma)(\alpha)}{\sigma(\alpha)}F(\alpha)(x)
\]
where
$(b\odot \gamma)(\alpha)= \langle  b\tensor \gamma, \Delta_\AAA(\alpha) \rangle$
and $\Delta_\AAA\from \RR \langle \AAA \rangle \to \RR\langle  \FF \rangle\tensor \RR\langle
\AAA \rangle$
is a left comodule map defined as follows:
For a graph $\alpha$, cut edges that are not included in the cycles to obtain
connected components, some of which are trees and some that are aromas, then sum
over all possible choices of edges to cut.
\label{lem: comp}
\end{lemma}

\begin{example}
 \[
    \Delta_{\AAA} \big( \centerfigure{\tseventeen } \big) = \one \tensor \centerfigure{\tseventeen }+ \ab \tensor \centerfigure{\tnine }
 \]
\end{example}

\subsubsection{The kernel of $F$}
For a given vector field $f$, $F$ will send some aromas or
linear combinations of aromas to zero. We call this the \emph{kernel of $F$}.
\begin{equation}
\ker F = \left\{k\in \RR\langle\AAA \rangle^\ast \text{ such that } \sum_{\alpha\in \AAA}\frac{h^{|\alpha|}k(\alpha)}{\sigma(\alpha)}F(\alpha)=0\right\}
\label{eq: kerdef}
\end{equation}
$\ker F$ is a description of functional dependencies between the derivatives of $f$.

Part of $\ker F$ is induced by the dimension $d$ of the surrounding space.
For example, when $d=1$, $F(\alpha)$ only depends on the number of
vertices in $\alpha$ with indegree
$0,1,2,\dotsc,$ and is independent of the exact arrangement of edges.
As an example in one dimension
$$F\big(\centerfigure{\tone\tone}\big)= F\big(\centerfigure{\tthree})=(f')^2.$$

In this article, we only consider quadratic vector fields.
These vector fields have all third derivatives equal to zero.
As a consequence, $F(\alpha)=0$ for all aromas containing a node with indegree
larger than or equal to 3.

Finally, specific classes of vector fields can have a larger kernel. 
For example, if $f$ is a divergence-free vector field, then $F(\alpha)=0$ is zero for all aromas that contain $\tone$ as a subgraph. 
Or if $f$ is an Hamiltonian vector field, then $F(\alpha)=0$ when $\alpha$ is
a cyclic aroma with an odd number of vertices.

The kernel of $F$ turns out to be crucial for the existence of solutions to
\eqref{darbouxdef} in the vector space spanned by aromas.

\subsubsection{Girard--Newton formula}
Recall that for the Kahan map,
\[
  \det D\Phi_h(x) = \frac{\det (I+\frac{h}{2}f'(x'))}{\det(I-\frac{h}{2}f'(x))},
\]
holds.

The expression $N_{uh}(x)=\det(I+uhf'(x))$ can be written as an aromatic series by
means of combinatorial formulas for symmetric polynomials, related to the
Girard--Newton formula \cite[Chapter 7]{stanley99ec2}.
\begin{theorem}
  The expression
 $N_{uh}(x)=\det(I+uhf')$ can be written as an aromatic series
 \begin{equation}
 \begin{aligned}
   \det(I+uhf') &= \sum_{\alpha\in \AAA} \frac{h^{|\alpha|}\eta_u(\alpha)}{\sigma(\alpha)}F(\alpha)\\
                &=1   + uhF(\cfig{\tone}) + \frac{u^2 h^2}{2} \left(F(\cfig{\tone\tone})-F(\cfig{\tthree})\right)+\dotsb
    \end{aligned}
   \label{eq: Newton}
 \end{equation}
In the above expression, $\eta_u$ is only non-zero if $\alpha$ is a product of
cyclic aromas, in which case
$\eta_u(\alpha)=\sgn(\pi_\alpha)u^{|\alpha|},$
where $\pi_\alpha$ is the permutation on $|\alpha|$ elements defined by the
graph $\alpha$.

Furthermore, for a fixed dimension $d$, the terms of the aromatic series
\eqref{eq: Newton} with $|\alpha|>d$ sum to zero.
\label{thm: girardnewton}
\end{theorem}

\begin{proof}
 Write $\det(I+uhf')= \prod_{i=1}^d(1+uh\lambda_i)$, where $\lambda_i$ are
 eigenvalues of $f'$.
 The first claim follows from writing out
 \[\det(I+uhf')=1+uh\sum_i \lambda_i + u^2h^2\sum_{i<j} \lambda_i\lambda_j +
   \dotsb + u^dh^d \prod_i \lambda_i,\]
 and applying \cite[Proposition 7.7.6]{stanley99ec2} to each term.
 The second claim follows from the fact that $\prod_{i=1}^d(1+uh\lambda_i)$ is a polynomial in $uh$ of degree $d$.
\end{proof}

\subsection{Darboux polynomials and aromatic series} \label{subsec:DarbouxAromatic}

We make the ansatz that $P$ can be written as an aromatic series $P=B(\gamma)$. 

By combining Lemma \ref{lem: mult}, Lemma \ref{lem: comp} and Theorem \ref{thm: girardnewton}, the
left hand side of \eqref{darbouxkahan} can then also be written as an aromatic series.

If $P(x)=B(\gamma)=\sum_{\alpha\in
  \AAA}\frac{h^{|\alpha|}\gamma(\alpha)}{\sigma(\alpha)}F(\alpha)$, then
\begin{equation}
N_{-\frac{h}{2}}(x)\cdot P\big( \Phi_h(x)\big)-P(x)\cdot N_{\frac{h}{2}}\big( \Phi_h(x)\big)= \sum_{\alpha\in
  \AAA}\frac{h^{|\alpha|}\langle Q(\gamma), \alpha \rangle}{\sigma(\alpha)}F(\alpha)
\label{eq: Qmeaning}
\end{equation}
where
\begin{equation}
\langle Q(\gamma), \alpha\rangle= \big\langle \eta_{-\frac{1}{2}}\tensor \phi \tensor \gamma- \gamma\tensor \phi
\tensor \eta_{\frac{1}{2}}, (I \tensor \Delta_\AAA )\circ \Delta_\sqcup(\alpha)  \big\rangle
\label{eq:Qdef}
\end{equation}

If $\langle Q(\gamma),\alpha\rangle$ were 0 for all $\alpha$, then $B(\gamma)$ would be a Darboux polynomial for all $f$.
This requirement is very strict and only satisfied  when $B$ is zero for all non-empty aromas, i.e. $B(\gamma)$ is a constant function.

We are looking for Darboux polynomials for a given $f$, in which case, the requirement is
\begin{equation}
  Q(\gamma)\in \ker F
  \label{eq:abstractequation}
\end{equation}
where $\ker F$ is as in \eqref{eq: kerdef}.

This means the set of aromatic series Darboux polynomials is a function of $\ker F$, that is the
\emph{linear dependencies between the aromas of $f$}.

In the following, we will develop necessary conditions on $\ker F$ such that \eqref{eq:abstractequation} has
nontrivial solutions.

\begin{example}
To illustrate the calculations of $Q(\gamma)$, we here display a detailed
calculation of 
$\langle Q(\gamma),\centerfigure{\tfive}\rangle$.

\begin{paragraph}{Step one:}
\[\Delta_{\sqcup}(\centerfigure{\tfive}) = \one \tensor \centerfigure{\tfive} +
  \centerfigure{\tfive }\tensor \one .
\]
\end{paragraph}
\begin{paragraph}{Step two:}
\[\Delta_{\AAA}(\one) = \one \tensor \one\quad \text{and}\quad  \Delta_{\AAA} \big(\centerfigure{\tfive }\big)=
  \one \tensor \centerfigure{\tfive }+ \ab \tensor \centerfigure{\tthree }.\]
Therefore
\begin{align*}
  (I \tensor \Delta_\AAA )\circ \Delta_\sqcup(\centerfigure{\tfive}) &= \one \tensor \Delta_{\AAA} \big(\centerfigure{\tfive }\big)+ \centerfigure{\tfive }\tensor \Delta_{\AAA}(\one)\\
                                                  &= \one \tensor \one \tensor \centerfigure{\tfive} + \one\tensor \ab \tensor \centerfigure{\tthree } + \centerfigure{\tfive } \tensor \one \tensor \one.
\end{align*}
\end{paragraph}
\begin{paragraph}{Step three:}
\begin{align*}
  \langle \eta_{-\frac{1}{2}}&\tensor \phi \tensor \gamma- \gamma\tensor \phi
  \tensor \eta_{\frac{1}{2}}, (I \tensor \Delta_\AAA )\circ \Delta_\sqcup(\alpha)  \rangle\\
 =&\big\langle\eta_{-\frac{1}{2}}\tensor \phi \tensor \gamma- \gamma\tensor \phi
 \tensor \eta_{\frac{1}{2}}, \one \tensor \one \tensor \centerfigure{\tfive } + \one\tensor \ab \tensor \centerfigure{\tthree } + \centerfigure{\tfive } \tensor \one \tensor \one  \big\rangle\\
 =&\eta_{-\frac{1}{2}}(\one)\phi(\one)\gamma( \centerfigure{\tfive })-\gamma(\one)\phi(\one)\eta_{\frac{1}{2}}( \centerfigure{\tfive })+
    \eta_{-\frac{1}{2}}(\one)\phi(\ab)\gamma( \centerfigure{\tthree })\\
  &-\gamma(\one)\phi(\ab)\eta_{\frac{1}{2}}( \centerfigure{ \tthree })+\eta_{-\frac{1}{2}}( \centerfigure{\tfive })\phi(\one)\gamma(\one)-\gamma( \centerfigure{\tfive })\phi(\one)\eta_{\frac{1}{2}}(\one)\\
 =& 1\cdot 1\cdot \gamma( \centerfigure{\tfive })-\gamma(\one)\cdot 1\cdot 0+1\cdot 1\cdot
 \gamma( \centerfigure{ \tthree})-\gamma(\one)\cdot 1\cdot (-\frac{1}{4})+0\cdot 1 \cdot
    \gamma(\one)\\
  &-\gamma( \centerfigure{\tfive })\cdot 1\cdot 1\\
 =& \gamma( \centerfigure{\tthree })+\frac{1}{4}\gamma(\one).
\end{align*}
\end{paragraph}
\end{example}
\begin{table}[htb]
\begin{tabular}{c|c}
  $\alpha$ & $\langle Q(\gamma), \alpha \rangle$\\
  \hline
  $\one$ & $0$\\
  $\tone$ & $-\gamma(\one)$\\
  $\ttwo$ & $\gamma( \centerfigure{\tone })-\frac{1}{2}\gamma(\one)$\\
  $\tthree$ & $0$\\
  $\tfour$ & $-2\gamma(\centerfigure{\tone})$ \\
  $\tfive$ & $\gamma(\centerfigure{\tthree })+\frac{1}{4}\gamma(\one)$\\
  $\tsix$ & $\gamma(\centerfigure{\ttwo})+\frac{1}{2}\gamma(\centerfigure{\tone})-\frac{1}{4}\gamma(\one)$\\
  $\tseven$ & $-\frac{1}{4}\gamma(\one)$\\
  $\teight$ & $\gamma( \centerfigure{\tfour})-\gamma(\centerfigure{\ttwo})-\gamma(\centerfigure{\tone})-\frac{1}{4}\gamma(\one)$\\
  $\tnine$ & $-\gamma(\centerfigure{\tthree})+\frac{1}{4}\gamma(\one)$\\
  $\tten$ & $-3\gamma(\centerfigure{\tfour})-\frac{1}{4}\gamma(\one)$
\end{tabular}\vspace*{3mm}
\caption{Calculations of $\langle Q(\gamma), \alpha \rangle $ for $|\alpha |\le 3$, see \eqref{eq: Qmeaning} and \eqref{eq:Qdef}.}
\label{tab: Qsmall}
\end{table}
Table \ref{tab: Qsmall} shows $\langle Q(\gamma), \alpha \rangle$ for all aromas
with $|\alpha|\le 3$.
Using the expressions in Table \ref{tab: Qsmall}, we can express the series
$B(Q(\gamma))$ in terms of $\gamma$ and $F$.

For $B(\gamma)$ to be a Darboux polynomial, $B(Q(\gamma))$ has to be equal to
zero, and specifically, each homogenous (in $h$) part has to equal zero.

In the following discussion, we assume that $B(\gamma)$ is a Darboux
polynomial and determine which consequences this has for $\gamma$ and $F$.

The Kahan map is self-adjoint.
As a consequence, linearly independent Darboux polynomials have to be symmetric
in $h$ (up to a sign change.)
We can therefore simplify our analysis by considering two disjoint classes of
$\gamma$:
\begin{enumerate}
  \item $\gamma$ that are non-zero only on $\alpha$ with an even number of
    vertices;
  \item $\gamma$ that are non-zero only on $\alpha$ with an odd number of
    vertices.
\end{enumerate}

\subsubsection*{Case 1: $\gamma$ non-zero for even $|\alpha|$}
We start with $\gamma$ that are non-zero only on $\alpha$ with even number of
vertices.

The
$\mathcal{O}(h)$ term in $B(Q(\gamma))$ is 
\[\big \langle  Q(\gamma), \centerfigure{\tone} \big\rangle F(\centerfigure{\tone}) = -\gamma(\one)F\big( \centerfigure{\tone }\big).\]
By our ansatz, this is equal to zero, and therefore
$\gamma(\one)=0$ or
$F( \centerfigure{ \tone })=\div f=0$.

This condition is an obvious consequence of the fact that the leading term of the Darboux polynomial defines a
preserved quantity for the continuous system. 
Specifically, there can only be Darboux polynomial with leading term $\gamma(\one)\ne 0$, if $\div f=0$, which is equivalent to $dx_1\wedge dx_2\wedge\cdots\wedge dx_n$ being a preserved quantity for the exact system.

For the $\mathcal{O}(h^2)$ term, we have
\[\begin{aligned}
 \big\langle  Q(\gamma),& \centerfigure{\ttwo}\big \rangle
 F(\centerfigure{\ttwo}) +\frac{1}{2}\big\langle Q(\gamma),
 \centerfigure{\tthree}\big\rangle F(\centerfigure{\tthree})\\
 &+\frac{1}{2}\big
 \langle Q(\gamma), \centerfigure{\tfour} \big \rangle F(\centerfigure{\tfour})
 \\
 =&\left( \gamma(\centerfigure{\tone})-\frac{1}{2}\right)F(\centerfigure{\ttwo})+0+\frac{1}{2}\left(-2\gamma(\centerfigure{\tone})\right)F(\centerfigure{\tfour})\\
 =&-\frac{1}{2}\gamma(\one)F(\centerfigure{\ttwo}),
 \end{aligned}
\]
where the final equality is due to our assumption that $\gamma(\alpha)$ is zero
when $|\alpha|$ is odd.
By our ansatz, this expression has to be equal to zero, which indicates that
either $\gamma(\one)=0$ or $F(\centerfigure{\ttwo})=D(\div f)\cdot
f=0.$

If $\div f=0$, then its derivatives are also zero, so this condition is already contained in the condition $\gamma(\one)F( \centerfigure{ \tone })=0$.

For the $\mathcal{O}(h^3)$ term, we consider the two subcases separately.

\subsubsection*{Subcase 1a: $\gamma$ non-zero for even $|\alpha|$, $\div f=0$.}

First, assume $F(\centerfigure{\tone})=0$, and (without loss of generality) that
$\gamma(\one)=1$.

All aromas containing $\tone$ as a subgraph are in $\ker F$, so we can
disregard all these terms.

With this simplification, the third order term in $B(Q(\gamma)$ is
\begin{equation}
  \begin{aligned}
    \big \langle Q(\gamma),&\centerfigure{\tfive}\big \rangle F(\centerfigure{\tfive})+\frac{1}{3}\big\langle Q(\gamma),\centerfigure{\tseven}\big \rangle F(\centerfigure{\tseven})\\
    =&\left(\gamma\big(\centerfigure{\tthree}\big)+\frac{1}{4}\right)F\big(\centerfigure{\tfive}\big)-\frac{1}{12}F\big(\centerfigure{\tseven}\big)
  \end{aligned}
\end{equation}
This can only be zero if there is a linear dependence between $F\big(\centerfigure{\tfive}\big)$ and $F\big(\centerfigure{\tseven}\big)$, specifically $F\big( \centerfigure{\tseven }\big)=\alpha F\big(\centerfigure{\tfive }\big)$ for some $\alpha\in \RR.$

We sum up the observation in a theorem.
\begin{theorem}
  If $\div f=0$ and there exists a Darboux polynomial in aromas $B(\gamma)$ with $\gamma(\one)=1$, then
  \[
    F\big( \centerfigure{\tseven }\big)=\alpha F\big(\centerfigure{\tfive }\big).
  \]
  for some constant $\alpha\in \RR$.
  Furthermore, either both
  $F\big(\centerfigure{\tseven}\big)$ and $F\big(\centerfigure{\tfive}\big)$ are equal to zero, or
  \[
    \gamma\big(\centerfigure{\tthree}\big)=\frac{\alpha-3}{12}.
  \]
  \label{thm: cond1}
\end{theorem}

\subsubsection*{Subcase 1b: $\gamma$ non-zero for even $|\alpha|$, $\gamma(\one)=0$.}
We now consider the case where $\gamma(\one)=0$.
By the expressions in Table \ref{tab: Qsmall}, the third order term is
\begin{equation}
  \begin{aligned}
    &\gamma\big(\centerfigure{\ttwo}\big)\left( F\big(\centerfigure{\tsix}\big)
    -F\big(\centerfigure{\teight}\big)\right)\\
  &+\gamma\big(\centerfigure{\tthree}\big)\left( F\big(\centerfigure{\tfive}\big)
    -\frac{1}{2}F\big(\centerfigure{\tnine}\big)\right)\\
  &+\gamma\big(\centerfigure{\tfour}\big)\left( F\big(\centerfigure{\teight}\big)
    -\frac{1}{2}F\big(\centerfigure{\tten}\big)\right).
\end{aligned}
\label{eq: fcond3}
\end{equation}
The expression is equal to $D(p)\cdot f-p\div f$, where
\[
p=\gamma\big(\centerfigure{\ttwo}\big)F\big(\centerfigure{\ttwo}\big)+\frac{\gamma\big(\centerfigure{\tthree}\big)}{2}F\big(\centerfigure{\tthree}\big)+\frac{\gamma\big(\centerfigure{\tfour}\big)}{2}F\big(\centerfigure{\tfour}\big).
\]
And we see that \eqref{eq: fcond3} is equal to zero if and only if
$\frac{dx_1\wedge dx_2\wedge\cdots\wedge dx_n}{p}$ is a preserved measure for
the continuous system $\dot{x}=f(x)$.

This holds in general: If $B(\gamma)$ is a solution to \eqref{darbouxdef}, then
the leading term of $B(\gamma)$ defines a preserved measure of the continuous system.

\subsubsection*{Case 2: Case 1: $\gamma$ non-zero for odd $|\alpha|$}
We now consider the case where $\gamma$ is nonzero only for $\alpha$ with odd
number of vertices.

In this case the $\mathcal{O}(h)$ term is automatically zero. 
For the $\mathcal{O}(h^2)$ term, we get
\begin{equation}
\gamma(\centerfigure{\tone})\left(  F(\centerfigure{\ttwo})-F(\centerfigure{\tfour})\right)=0.
\label{eq: Oh2}
\end{equation}
This implies that either $\gamma(\centerfigure{\tone})=0$ or
\begin{equation}
  F(\centerfigure{\ttwo})=F(\centerfigure{\tfour}).
\label{eq: fcond2}
\end{equation}
Equation \eqref{eq: fcond2} is equivalent to $\div f^2=D(\div f)\cdot f$
which again is equivalent to the measure
$\frac{dx_1\wedge
  dx_2\wedge\cdots\wedge dx_n}{\div f}$ being a preserved measure for the
continuous system $\dot{x}=f(x)$

For the $\mathcal{O}(h^3)$ term, assuming either
$\gamma(\centerfigure{\tone})=0$ or
$F(\centerfigure{\ttwo})=F(\centerfigure{\tfour})$ makes the term simplify to zero.

\section{Examples} \label{sec:examples}
In this section we shall demonstrate the suggested approach on a number of well-known examples. 
All experiments are performed with the computer algebra system Maple.
The algorithm takes any quadratic vector field as input and 
the Kahan map and its Jacobian determinant are computed.
Then the set of all aroma functions up to a specified order is calculated, where each aroma function is now a multivariate polynomial. Next, a maximal linearly independent subset is selected, say $\{\mathbf{e}_k\}_{k=0}^N$. Now, the sought Darboux polynomial is expressed as in the ansatz \eqref{dpansatz}, and the condition \eqref{darbouxdef} yields a system of linear equations for the coefficients $\{P_k\}_{k=0}^N$ of \eqref{dpansatz}.
One may find $d\geq 0$ solutions and thus $d$ linearly independent Darboux polynomials which are density functions of measures preserved by the Kahan method for the given quadratic vector field.

\subsection{Homogeneous Nambu system}
Consider the following system in $\mathbb{R}^3$, $\mathbf{x}=[x,y,z]^T$ analysed in \cite{celledoni14ipo},
$$
   \dot{\mathbf{x}} = \nabla H_1(\mathbf{x})\times\nabla H_2(\mathbf{x}),
$$
where $H_1$ and $H_2$ are quadratic homogeneous polynomials in $(x,y,z)$, i.e. for symmetric
$3\times 3$-matrices $A$ and $B$, one has $H_1(\mathbf{x})= \mathbf{x}^TA\mathbf{x}$ and
$H_2(\mathbf{x})= \mathbf{x}^T B \mathbf{x}$. $H_1$ and $H_2$ are first integrals of the system.
It was found in \cite{celledoni14ipo} that with
$$
     C=A\cdot \adj(B)\cdot A,\quad H_3(x) =\mathbf{x}^TC\mathbf{x},
$$
the corresponding Kahan map has a preserved measure with reciprocal density function
$$
\bar{g}_1 = (1+4h^2H_3)^2
$$
and two modified first integrals
$$
\tilde{H}_1=\frac{H_1}{1+4h^2 H_3},\quad
\tilde{H}_2=\frac{H_2}{1+4h^2 H_3}.
$$
Another reciprocal density function independent of $h$  is obtained as
$$
\bar{g}_2 = \bar{g}_1 \tilde{H}_1\tilde{H}_2 =  H_1 H_2
$$
and a $h$-independent first integral is obtained as $H_1/H_2$.

Using our method of Darboux polynomials and aromas, we find
that two linearly independent density functions are
\begin{align*}
   g_1 &= 1 - \frac{h^2}{12} F(\cfig{\tthree})+\frac{h^4}{96} F(\cfig{\televen}), \\[2mm]
   g_2&= F(\cfig{\tsixteen})-3 F(\cfig{\televen}),
\end{align*}
so their ratio is a first integral. It can be convenient to replace $g_1$ by
$\tilde{g}_{1}= g_1 + \frac{1}{288}h^4 g_2$, i.e.
$$
\tilde{g}_{1}= 1 - \frac{h^2}{12} F(\cfig{\tthree}) + 
\frac{h^4}{288} F(\cfig{\tsixteen}).
$$
For any divergence free vector field in $\mathbb{R}^3$ it holds that
$$
F(\cfig{\tsixteen}) = \frac{1}{2}\left( F(\cfig{\tthree})\right)^2
$$
so that in fact we have
$$
\tilde{g}_{1} = \left(1 - \frac{h^2}{24} F(\cfig{\tthree})\right)^2.
$$
It is possible to prove that $F(\cfig{\tthree})=32 \,x^TCx$ where $C$ is given as
$$
C=\adj(\adj(A)+\adj(B)) - \adj(\adj(A)) - \adj(\adj(B)) - B\cdot\adj(A)\cdot B- A\cdot \adj(B)\cdot A.
$$
Note that $\adj(\adj(A)) = \det(A)\cdot A$. In any case, we see here that $C$ is symmetric in the arguments $A$ and $B$, an issue discussed in \cite{celledoni14ipo}.

Regarding the second measure $g_2$, it is a homogeneous, quartic, $h$-independent polynomial.
Recalling above the $h$-independent quartic density function $\bar{g}_2=H_1H_2$ and the rational first integral
$\tilde{H}:=H_1/H_2$, it seems plausible that our $g_2$ is of the form
$$
  g_2 = \bar{g}_2\cdot P(\tilde{H}),\quad P(z) = \alpha z+ \beta + \frac{\gamma}{z}.
$$
A calculation in Maple shows that there is a unique solution for $\alpha, \beta, \gamma$, but their expressions are rather complicated.

\subsection{A generalised Lotka--Volterra system and the dressing chain}
Consider the problem
\begin{equation}
\begin{aligned}
\dot{x} &= x(\beta z-\gamma y), \\
\dot{y} &= y(-\alpha z+\gamma x), \\
\dot{z} &= z(\alpha y-\beta x).
\end{aligned}
\label{eq: LV}
\end{equation}
Clearly $I_0=x+y+z$ is a first integral of the continuous system for any choice of $\alpha, \beta, \gamma$. Since $I_0$ is linear, it is preserved by any RK method, in particular by the Kahan method.

\subsubsection{Divergence free case $\alpha =\beta=\gamma=1$}  \label{subsec:LVDF}
In this case, the ODE has a second preserved integral  $I_1=xyz$ .
Using the standard approach, we find the densities
\begin{align*}
g_1 =& 1 - \frac{h^2}{8} F(\cfig{\tthree}), \\
g_2=&4F(\cfig{\televen}) - 4F(\cfig{\ttwelve})
+F(\cfig{\tsixteen})
+h^2 \bigg(
F(\cfig{\tthirtyseven})\\&-F(\cfig{\tthirtysix}) 
+F(\cfig{\tfortythree})-\frac14 F(\cfig{\tfiftyseven}) 
\bigg),
\end{align*}
so this allows for one first integral. However, augmenting the basis of aroma functions with e.g.
$I_0 F(t)$ for $|t|\leq 6$, we can get two extra density functions
$$
g_3=I_0 g_1,\quad g_4=I_0 g_2,
$$
and we actually recover the general integral $I_3=I_0=\frac{g_3}{g_1}$ yielding integrability.

The dressing chain system reads
\begin{align*}
\dot{x}  &= -y^2+z^2-b+c, \\ 
\dot{y}&=  x^2-z^2+a-c, \\
\dot{z}&= -x^2+y^2-a+b,
\end{align*}
where $a, b$ and $c$ are free parameters. Invariants of this system are
\begin{align*}
J_1 &= x+y+z, \\
J_2 &= (x+y)(y+z)(z+x)-ax-by-cz.
\end{align*}
We can do the same basis augmentation as before, adding
functions of the form $J_1 F(t)$, and we obtain preserved measures with densities
\begin{align*}
  g_1 =& 1-\frac{h^2}{8} F(\cfig{\tthree}), \\
  g_2 =& 4 F(\cfig{\televen}) - 4  F(\cfig{\ttwelve})  +  F(\cfig{\tsixteen}) 
  +h^2\bigg( F(\cfig{\tthirtyseven}) \\
  & -F(\cfig{\tthirtysix}) + F(\cfig{\tfortythree}) -\frac14 F(\cfig{\tfiftyseven})\bigg),\\
  g_3 =& J_1 g_1, \\
  g_4 =& J_1 g_2.
\end{align*}
It is interesting to note that the aroma expressions are exactly the same as in the Lotka-Volterra divergence free example. The reason is that these two systems are linked via a linear transformation, see Lemma~\ref{lemma36}.

With $\alpha=\beta=\gamma=1$ in \eqref{eq: LV}, we have the transformation
$$
\begin{array}{lcl}
x &=& \tilde{x}+\tilde{z} \\
y &=& \tilde{x}+\tilde{y} \\   
z &=& \tilde{y}+\tilde{z}
\end{array}
\quad \Rightarrow \quad 
\begin{array}{lcl}
\dot{\tilde{x}} &=& -\tilde{y}^2+\tilde{z}^2 \\
\dot{\tilde{y}} &=&  \tilde{x}^2-\tilde{z}^2 \\   
\dot{\tilde{z}} &=& -\tilde{x}^2+\tilde{y}^2.
\end{array} 
$$

\subsubsection{The case $\alpha=\beta=1, \gamma=-1$}
This is the following particular case of \eqref{eq: LV}
\begin{equation}
\begin{aligned}
\dot{x} &= x(y + z), \\
\dot{y} &= -y( x + z), \\
\dot{z} &= z( y - x).
\end{aligned}
\label{eq: LVspes}
\end{equation}
which is also investigated in \cite{celledoni19udd}.
We get 
\begin{align*}
g_1 =& -2 F(\cfig{\tthree})+ F(\cfig{\tone}\cfig{\tone}), \\
g_2 =& 4 F(\cfig{\ttwelve}) - 2F(\cfig{\tsixteen})-2F(\cfig{\tone}\cfig{\tfive})+F(\cfig{\tone}\cfig{\tseven}), \\
g_3 =& 2 F(\cfig{\televen}) - 2 F(\cfig{\tone}\cfig{\tfive}) +  F(\cfig{\tone}\cfig{\tone}\cfig{\ttwo}), \\
g_4 =& -8F(\cfig{\tthirtysix}) + 8F(\cfig{\tfortythree})
-2F(\cfig{\tfiftyseven})  +4  F(\cfig{\tone}\cfig{\ttwentysix})\\
& - 4F(\cfig{\tone}\cfig{\ttwentyeight})+F(\cfig{\tone}\cfig{\tthirtyfive}) ,\\
g_5=& 2 F(\cfig{\tthirtysix}) - 4 F(\cfig{\tthirtyeight}) + 4 F(\cfig{\tone}\cfig{\ttwentysix}) 
- 2 F(\cfig{\tone}\cfig{\ttwentyeight})  \\
&+ F(\cfig{\tone}\cfig{\tone}\cfig{\ttwelve}) - 2 F(\cfig{\tone}\cfig{\tone}\cfig{\tfourteen}).
\end{align*}
In this case one may form first integrals e.g. by
$$
   I_i = \frac{g_{i+1}}{g_1},\quad i=1,2,3,4,
$$
but only two of them, e.g. $I_1, I_2$, are functionally independent. One checks that $I_1=(x+y+z)^2=I_0^2$ (the square of the general invariant given above).
All the measures are independent of the step size $h$. We write out the list of densities and invariants:
\begin{align*}
g_1 &= -4z^2, \\
g_2 &= g_1 I_1, \\
g_3 &=16 xy (x+z)(y+z), \\
g_4 &= g_1 I_1^2, \\
g_5 &= g_3 I_1,\\[1mm]
I_2 &= - \frac{4xy(x+z)(y+z)}{z^2}, \\
I_3 &= I_1^2, \\
I_4 &= I_1 I_2.
\end{align*}


\subsection{The generalised Ishi problem}
We find in \cite{celledoni14ipo} the problem
\begin{align*}
\dot{x}  &= -c_2 x + b_2y + b_3 z,  \\ 
\dot{y}&= c_1x + c_2 y + c_3 z, \\
\dot{z}&=  a_{11}x^2 + a_{12} xy + a_{22} y^2.
\end{align*}
The Kahan method preserves volume exactly if the parameters satisfy two conditions, see \cite{celledoni14ipo}.
Alternatively, one may express the parameters $a_{ij}$  as
\begin{equation}\label{eq:aeq}
a_{11} = k A_2c_3,\quad a_{12} = -k(A_1c_3+A_2b_3),\quad a_{22}=k A_1 b_3,
\end{equation}
where $k$ is an arbitrary parameter.  Here 
 \begin{equation*}
 A_1=b_2c_3-b_3c_2,\quad A_2=c_2c_3+b_3c_1,\quad A_3=-(b_2c_1+c_2^2).
 \end{equation*}
 The continuous system thus obtained has the invariants
 \begin{align} 
H_1 &= z+\frac{k}2(c_3\,x-b_3\, y)^2,   \label{eq:genexinv1} \\
H_2&=\frac{k}3\, \left( c_{3}x-b_{3}y \right)^{3}+\frac{c_1}{2}\,{x}^{2}+c_{2}\,xy+c_{3}\,xz-\frac{b_2}{2}\,{y}^{2}-b_{3}\,yz.
 \label{eq:genexinv2}
\end{align}
Preserved invariants of the Kahan map have been shown to be
 \begin{align*}
 \widetilde{H}_1 &= z + \frac{k}2(c_3\,x-b_3\, y)^2 - \frac{k h^2}{8}\big( A_2\,x - A_1\,y \big)^2, \\
 \widetilde{H}_2 &= H_2 + \frac{h^2}{24}\bigg(A_3(-c_1 x^2-2c_2xy-2c_3xz+2b_3yz+b_2y^2)+(A_1c_3-A_2b_3)z^2\\
 & +k(-2c_3A_2^2 x^3+2A_2(b_3A_2+2c_3A_1)x^2 y -2A_1(c_3A_1+2b_3A_2)xy^2+2b_3A_1^2 y^3)\bigg),
 \end{align*}
 see \cite{celledoni14ipo}.
Using the aroma approach, we find the parameter independent measures
\begin{align*}
g_1=& 1, \\
g_2=& F(\cfig{\tsixteen}) - 4 F(\cfig{\ttwelve})-\frac{h^2}{2} F(\cfig{\tfortytwo}).
\end{align*}
This shows that the volume is exactly preserved by Kahan's method and there is another measure with density $g_2$ also preserved, and $g_2$ is therefore also a first integral. It turns out that $g_2$ only depends on
$ \widetilde{H}_1$, we have in fact
$$
g_2 = 2 A_3^2 + 4 k (A_1c_3-A_2 b_3)^2 \widetilde{H}_1
$$


\subsection{The inhomogeneous Nambu system}

We consider two quadratic functions
\begin{align*}
H(\mathbf{x}) &= \mathbf{x}^T\mathbf{H}\mathbf{x} + \mathbf{h}^T\mathbf{x}, \\
K(\mathbf{x}) &= \mathbf{x}^T\mathbf{K}\mathbf{x} + \mathbf{k}^T\mathbf{x} ,
\end{align*}
where $\mathbf{H}$ and $\mathbf{K}$ are arbitrary symmetric $3\times 3$ matrices and
 $\mathbf{h}$ and $\mathbf{k}$ are vectors in $\mathbb{R}^3$. The ODE we consider is
 \begin{equation} \label{inhomNambu}
 \dot{\mathbf{x}} = \nabla H(\mathbf{x})\times   \nabla K(\mathbf{x}).
 \end{equation}
A preserved measure of the corresponding Kahan map
can be found with density function
\begin{align*}
 g&=1-\frac{1}{12}h^2 F(\cfig{\tthree})\\
 &+h^4\left(\frac{1}{36} F(\cfig{\ttwelve})-\frac{1}{72}F(\cfig{\televen})-\frac{1}{96}F(\cfig{\tsixteen})\right)\\
& +\frac{1}{384}h^6\left(
 F(\cfig{\tfortyseven})-2\,F(\cfig{\tthirtyeight})
 \right).
 \end{align*}
 The general expression for the density function has 15806 terms.
 Since the aromatic functions are not generally linearly independent, one has several equivalent representations, in the sixth order term we have for instance 
 \begin{align*}
  F(\cfig{\tfortyseven})-2\,F(\cfig{\tthirtyeight}) &=
\frac25\left(  F(\cfig{\tfortytwo})-3\,F(\cfig{\tthirtyeight}) \right) \\
&= \frac13\left(  F(\cfig{\tforty})-4\,F(\cfig{\tthirtyeight}) \right).
 \end{align*}
 Similarly, the fourth order term could be replaced by
 $$
 h^4\left(\frac{1}{12}F(\cfig{\tthirteen})-\frac{1}{24}F(\cfig{\televen})-\frac{1}{96}F(\cfig{\tsixteen})\right).
 $$
Note that the Kahan map of some special cases of the inhomogeneous Nambu system (4.4) for which the density factorises, were treated in \cite{celledoni19gai,celledoni22dad}.

\subsection{Homogeneous quadratic  divergence free systems in  $\mathbb{R}^3$ }
Finally, we consider the general case of quadratic homogeneous divergence free vector fields in 3 dimensions.
They can be written in the form
\begin{equation}
\dot{\mathbf{x}}=
\left(
\begin{array}{l}
\mathbf{x}^TA\mathbf{x} \\ \mathbf{x}^TB\mathbf{x} \\ \mathbf{x}^TC\mathbf{x}
\end{array}
\right)
\label{eq: ABC}
\end{equation}
for symmetric matrices $A, B$ and $C$, where $\mathbf{x}=(x,y,z)^T$. 

It can be shown that \eqref{eq: ABC} is divergence free if and only if
$$\left\{ a_{{1,1}}+b_{{1,2}}+c_{{1,3}}=0,\ a_{{2,1}}+b_{{2,2}}+c_{{2,3}}=0
,\ a_{{3,1}}+b_{{3,2}}+c_{{3,3}}=0 \right\}.$$

Based on multiple experiments with randomized $A, B$ and $C$, we pose the following conjecture:
\begin{conjecture}
Let $f$ be a quadratic, homogeneous, divergence free vector field in $\mathbb{R}^3$. Assume that
$F(\cfig{\tfive}) \ne 0$ and that
\begin{equation}\label{eq:GB}
     F(\cfig{\tseven}) = \alpha F(\cfig{\tfive})
\end{equation}
holds for some $\alpha\in \RR$. Then there is a preserved measure with density 
\begin{equation} \label{conjecture}
   g = 1 - \frac{3-\alpha}{24} h^2 F(\cfig{\tthree}) + \frac{\alpha^2}{24(3+\alpha)} h^4 F(\cfig{\televen}).
\end{equation}
\end{conjecture}

\begin{remark}
It might seem like a reasonable idea to try to prove the conjecture simply by brute force using a CAS such as Maple.
However, the calculations become intractable using the algorithm described in Section~1 when all the matrix elements are treated as free parameters.
One can prove that under the hypotheses of the conjecture 
$$
    F(\cfig{\tthirteen}) = \frac{\alpha}{3} \left(F(\cfig{\televen})+F(\cfig{\ttwelve})\right).
$$
Several test examples indicate that also the following relation holds
$$
F(\cfig{\ttwelve}) = 2\,F(\cfig{\televen}).
$$
These two relations are central in deriving the formula \eqref{conjecture}.
\end{remark}

\textbf{Acknowledgements.} The authors would like to thank the anonymous referees for helpful comments and suggestions. The authors would like to thank the Isaac Newton Institute for Mathematical Sciences, Cambridge, for support and hospitality during the programme \textit{Geometry, compatibility and structure preservation in computational differential equations} where work on this paper was undertaken. This work was supported by EPSRC grant no EP/R014604/1. Celledoni, McLachlan and Quispel acknowledge the Simons Foundation. EC and BO have received funding from the European Union’s Horizon 2020 research and innovation programme under the Marie Sk{\l}odowska-Curie grant agreement No 860124.

\providecommand{\bysame}{\leavevmode\hbox to3em{\hrulefill}\thinspace}
\providecommand{\MR}{\relax\ifhmode\unskip\space\fi MR }
\providecommand{\MRhref}[2]{%
  \href{http://www.ams.org/mathscinet-getitem?mr=#1}{#2}
}
\providecommand{\href}[2]{#2}

\end{document}